\documentclass[11pt]{article}
\usepackage{epsfig}
\usepackage{amssymb,amsmath,amsthm,amscd}
\usepackage{latexsym}

\newtheorem{thm}{Theorem}[section]
\newtheorem{prop}[thm]{Proposition}

\newtheorem{lem}[thm]{Lemma}

\theoremstyle{definition}
\newtheorem{defn}[thm]{Definition}

\newcounter{labelflag} 

\newcommand{\Label}[1]{
                       \ifnum\thelabelflag=1
                          \ifmmode
                             \makebox[0in][l]{\qquad\fbox{\rm#1}}
                          \else
                             \marginpar{\vspace{0.7\baselineskip}
                                        \hspace{-1.1\textwidth}
                                        \fbox{\rm#1}}
                          \fi
                       \fi
                       \label{#1} }

\newcommand{\be}{\begin{equation}}
\newcommand{\ee}{\end{equation}}

\newcommand{\R}{\mathbb{R}}
\newcommand{\N}{\mathbb{N}}

\def \calf {{  {\mathcal{F}} }}

\def \calftwo {{  {\mathcal{F}}_2  }}
\def \cala {{  {\mathcal{A}}  }}
\def \calb {{  {\mathcal{B}}  }}
\def \cald {{  {\mathcal{D}}  }}
\def \caln {{  {\mathcal{N}}  }}

\def \tomega { {\tilde{\omega} }}

\def \tilv { { \tilde{v} }}

\def \thonet {{  \theta_{1,t}  }}
\def \thtwot {{  \theta_{2,t}  }}

\begin{document}

\baselineskip=1.5\baselineskip

\begin{titlepage}
\title{  \bf
Periodic Random  Attractors for Stochastic  Navier-Stokes Equations
on Unbounded Domains }
\vspace{10mm}

\author{
Bixiang Wang 
\vspace{5mm}\\
Department of Mathematics\\
 New Mexico Institute of Mining and
Technology  \\ Socorro,  NM~87801, USA  \vspace{3mm}\\
Email: bwang@nmt.edu }
\date{}

\end{titlepage}

\maketitle

\medskip

\begin{abstract}
This paper is concerned with the asymptotic behavior of solutions
of the two-dimensional Navier-Stokes equations 
with both non-autonomous deterministic and
stochastic terms defined on unbounded domains.
We first introduce a continuous cocycle   for the equations  
and then prove   the existence
and uniqueness of tempered  random  attractors.
We also characterize the structures of the random  attractors by complete
solutions.  When deterministic forcing terms are periodic, we show
that the tempered  random attractors  are also periodic.
Since the Sobolev embeddings on unbounded domains  are not compact,
we establish the pullback asymptotic compactness of solutions  by
Ball\rq{}s idea of  energy equations.
\end{abstract}

complete solution.


\section{Introduction}
\setcounter{equation}{0}
 
 In this paper,    we investigate   the 
 pullback   attractors  for the 
 two-dimensional Navier-Stokes equations on unbounded domains
 with  non-autonomous deterministic  as well   as 
stochastic terms. 
Let $Q$ be an unbounded open set in $\R^2$ with boundary
$\partial Q$.
 Given   $\tau \in\R$,   consider      the stochastic 
 Navier-Stokes equations with multiplicative noise:
 \be
 \label{intro1}
 {\frac {\partial u}{\partial t}}
 -\nu \Delta u  + (u \cdot \nabla) u
 = f(x,t) -\nabla p
 +  \alpha u \circ {\frac {d w}{dt}}, \quad x \in Q \ \ \mbox{ and } \ t >\tau, 
 \ee
 \be\label{intro2}
 {\rm  div }\  u =0,
  \quad x \in Q \ \ \mbox{ and } \ t >\tau, 
 \ee
 together with  homogeneous  Dirichlet boundary condition, 
 where $\nu, \ \alpha \in \R$
 with  $\nu>0$,  $f$ is a 
 given function defined on $Q \times \R$, 
 and $w$  
 is  a  two-sided real valued Wiener process 
 defined in a 
 probability space. 
 The  stochastic  equation \eqref{intro1}  is 
 understood in the sense of Stratonovich  integration. 

The attractors of the Navier-Stokes equations have been extensively
studied  in the literature, see, e.g., 
\cite{bab1, bal1, car3, car4, hal1, ros1, sel1, tem1}
for deterministic   equations
and  \cite{ cra2, fla1, schm1}
for stochastic  equations.  
Particularly, 
in the deterministic case (i.e.,  $\alpha =0$),   the autonomous global  attractors
and 
the non-autonomous pullback attractors
of     \eqref{intro1}-\eqref{intro2} on {\it unbounded}  domains
have been studied in \cite{ros1} and \cite{car3, car4}, respectively.
For the  stochastic equations with additive noise  
and  time-independent  $f$,  the asymptotic compactness
of solutions  on {\it unbounded}
  domains has been investigated in \cite{brz1}.
As far as the author is aware,  there is no  result available in the literature
on the existence of random attractors for the stochastic equations
\eqref{intro1}-\eqref{intro2}   with   {\it  time-dependent}  $f$ 
even on bounded domains.
The purpose of the present 
 paper is  to  investigate  this problem 
 and examine the periodicity of random
attractors  when $f$ is periodic in time.

 It is  worth mentioning  that 
 the concept of pullback attractors 
  for random  systems 
  with time-independent $f$
  was introduced  in \cite{cra2, fla1, schm1}  and
  the existence of such attractors for 
  compact systems was proved  in  
     \cite{arn1,  car2, chu2, cra1, cra2, fla1, huang1, 
    kloe1, schm1} and the references therein.
    For non-compact systems,  the existence of pullback attractors
    was established in \cite{bat1, bat2,  wan2, wan3}.
    In the present paper,    we study pullback   attractors
    for the stochastic equations \eqref{intro1}-\eqref{intro2}
    on unbounded domains with time-dependent $f$.
    In this case,  the random dynamical  systems
    associated with the equations are non-compact.

To deal with  the  stochastic equations with 
 non-autonomous $f$,   we need
to combine  the  ideas of non-autonomous deterministic 
dynamical systems  and  that of random dynamical systems.
Particularly,   the concept of dynamical systems defined  over 
two parametric spaces, say $\Omega_1$ and $\Omega_2$,   is needed,
where $\Omega_1$ is 
a nonempty set
used to deal with the non-autonomous deterministic terms,
and $\Omega_2$ is  a probability space  responsible for the stochastic terms.
The existence and uniqueness of random attractors 
for dynamical systems over two parametric spaces have
been recently established in \cite{wan4}.
For the stochastic Navier-Stokes equations
\eqref{intro1}-\eqref{intro2}, we may take
$\Omega_1$   as the set of all    translations of $f$.
We   can also take  $\Omega_1$ as  the collection
of all initial times, i.e., $\Omega_1 =\R$. 
In this paper, we will   choose  $\Omega_1 =\R$.
We  first  define  a  continuous 
cocycle   for \eqref{intro1}-\eqref{intro2}
over $\Omega_1$   and $\Omega_2$, and then
prove the existence of tempered random absorbing sets.
Since the Sobolev embeddings on unbounded domains are no 
longer compact, we have to appeal   to the idea of energy equations
to establish   the pullback asymptotic compactness of solutions.
This  method    was introduced by
Ball in \cite{bal1} for deterministic equations, and used by the authors
in \cite{car3, car4, ros1}    for  the deterministic Navier-Stokes equations
on unbounded domains and in \cite{brz1}   for  the
stochastic equations with time-independent $f$. 
We will adapt this approach to the stochastic equations
\eqref{intro1}-\eqref{intro2}   with time-dependent $f$,
and prove the existence of tempered random attractors 
 for the equations.
We also consider
the random attractors  in   the case where $f$ is a periodic function in time.
If $f$ is periodic, we will show   that the  tempered random attractors
are also periodic in  some sense.
Following \cite{wan4}, the structures of the
 tempered random attractors  will  be characterized by
  the tempered complete solutions.

 In the next   section,  we will recall some results
 on pullback attractors  for random dynamical 
 systems over two parametric spaces.  A continuous
 cocycle for the stochastic Navier-Stokes equations
 \eqref{intro1}-\eqref{intro2} with non-autonomous $f$
 is defined in Section  3.  We   then
  derive uniform estimates
 of the  solutions in Section 4 and prove the existence
 and uniqueness  of pullback
 attractors in Section 5.

In    the  sequel,    we  will use  
$\| \cdot \|$ and $(\cdot, \cdot)$ to denote  the norm and the inner product
of $L^2(Q)$,  respectively.   The
norm of a       Banach space $X$  is generally  written as    $\|\cdot\|_{X}$.
   The letters $c$ and $c_i$ ($i=1, 2, \ldots$)
are  used  to denote   positive constants 
whose  values are not significant in the context.
   
\section{Theory of Pullback Attractors}
\setcounter{equation}{0}

In this section,  we recall some   results 
on pullback   attractors    for random dynamical
systems  with two parametric spaces   as presented
in \cite{wan4}. This sort of dynamical systems can be generated
by differential  equations  with both deterministic
and stochastic  non-autonomous
external terms. 
All results given in this section are not original and they
are presented here   just   for the reader's   convenience.
We also  refer   the reader to
  \cite{bat1, cra1, cra2, fla1, schm1} for the theory of pullback
attractors  for random dynamical  systems with one parametric
space.

Let $\Omega_1$ be a nonempty  set
and $\{\thonet\}_{t \in \R}$  be a family of
mappings   from $\Omega_1$
into itself   such that
 $\theta_{1, 0}  $ is the
identity on $\Omega_1$
and $\theta_{1,  s+t}  = \theta_{1,t,}
  \circ \theta_{1,s}  $ for all
$t, s \in \R$. 
Let $(\Omega_2, \calftwo, P)$
be a probability space and 
 $\theta_2 : \R \times \Omega_2 \to \Omega_2$
 be  a    $(\calb (\R) \times \calftwo, \calftwo)$-measurable mapping
 such that $\theta_2(0,\cdot) $ is the
identity on $\Omega_2$, $\theta_2 (s+t,\cdot) = \theta_2 (t,\cdot) \circ \theta_2 (s,\cdot)$ for all
$t, s \in \R$
and $P \theta_2 (t,\cdot)  =P$
for all $t \in \R$.
We usually  write 
$\theta_2 (t, \cdot)$ as $\thtwot$
and   call  both
$(\Omega_1, \{\thonet\}_{t \in \R})$ and 
$(\Omega_2, \calftwo, P,  \{\thtwot\}_{t \in \R})$
a parametric   dynamical system.

Let  $(X, d)$   be  a complete
separable  metric space with  Borel $\sigma$-algebra $\calb (X)$.
Given $r>0$     and $D \subseteq  X$,
the neighborhood 
of $D$  with radius   $r$ is written  as 
$\caln_r(D)$. 
 Denote by $2^X$ 
  the collection of all subsets of $X$. A set-valued mapping
$K: \Omega_1 \times \Omega_2 \to  2^X$   is called measurable
with respect to $\calftwo$
in $\Omega_2$
if  the value  $K(\omega_1, \omega_2)$
is a closed  nonempty subset  of $X$
for all $\omega_1 \in \Omega_1$ and $\omega_2 \in \Omega_2$,
 and  the mapping
$ \omega_2 \in  \Omega_2
 \to d(x, K(\omega_1, \omega_2) )$
is $(  \calftwo, \ \calb(\R) )$-measurable
for every  fixed $x \in X$ and $\omega_1 \in \Omega_1$.
If $K$ is  measurable  with respect to $\calftwo$
in $\Omega_2$,   then we   say    that      the family 
$\{K(\omega_1, \omega_2): \omega_1 \in \Omega_1, \omega_2 \in \Omega_2 \}$
  is measurable
with respect to $\calftwo$
 in $\Omega_2$.
 We now   define a cocycle on $X$ over two parametric spaces.

\begin{defn} \label{ds1}
 Let
$(\Omega_1,  \{\thonet\}_{t \in \R})$
and
$(\Omega_2, \calftwo, P,  \{\thtwot\}_{t \in \R})$
be parametric  dynamical systems.
A mapping $\Phi$: $ \R^+ \times \Omega_1 \times \Omega_2 \times X
\to X$ is called a continuous  cocycle on $X$
over $(\Omega_1,  \{\thonet\}_{t \in \R})$
and
$(\Omega_2, \calftwo, P,  \{\thtwot\}_{t \in \R})$
if   for all
  $\omega_1\in \Omega_1$,
  $\omega_2 \in   \Omega_2 $
  and    $t, \tau \in \R^+$,  the following conditions (i)-(iv)  are satisfied:
\begin{itemize}
\item [(i)]   $\Phi (\cdot, \omega_1, \cdot, \cdot): \R ^+ \times \Omega_2 \times X
\to X$ is
 $(\calb (\R^+)   \times \calftwo \times \calb (X), \
\calb(X))$-measurable;

\item[(ii)]    $\Phi(0, \omega_1, \omega_2, \cdot) $ is the identity on $X$;

\item[(iii)]    $\Phi(t+\tau, \omega_1, \omega_2, \cdot) = \Phi(t, \theta_{1,\tau} \omega_1,  \theta_{2,\tau} \omega_2, \cdot) \circ \Phi(\tau, \omega_1, \omega_2, \cdot)$;

\item[(iv)]    $\Phi(t, \omega_1, \omega_2,  \cdot): X \to  X$ is continuous.
    \end{itemize}
    
    If,  in addition,  there exists  a
    positive number   $T $ such that
    for every $t\ge 0$, $\omega_1 \in \Omega_1$  and $\omega_2 \in \Omega_2$,
$$
\Phi(t, \theta_{1, T} \omega_1, \omega_2, \cdot)
= \Phi(t, \omega_1,  \omega_2, \cdot ),
$$
then $\Phi$ is called  
a  continuous periodic  cocycle  on $X$ with period $T$.
\end{defn}

In the sequel, we use  $\cald$ to denote
 a  collection  of  some families of  nonempty subsets of $X$:
\be
\label{defcald}
{\cald} = \{ D =\{ D(\omega_1, \omega_2 ) \subseteq X: \ 
D(\omega_1, \omega_2 ) \neq \emptyset,  \ 
  \omega_1 \in \Omega_1, \
  \omega_2 \in \Omega_2\} \}.
\ee
Two elements  $D_1$ and $D_2$  of  $\cald$
are said  to be equal if 
$D_1(\omega_1, \omega_2) =  D_2(\omega_1, \omega_2)$
for any $\omega_1 \in \Omega_1$   and
$\omega_2 \in \Omega_2$.
Sometimes, we require  that $\cald$ 
is  neighborhood closed
which is defined as follows.

\begin{defn} 
\label{defepsneigh1}
A collection $\cald$ of some families 
of nonempty subsets of $X$
is said  to be   neighborhood closed if   for each
$D=\{D(\omega_1, \omega_2): 
\omega_1 \in \Omega_1, \omega_2 \in \Omega_2 \}
\in \cald$,   there exists a positive number
$\varepsilon$ depending on $D$ such that  the family
\be\label{defepsneigh2}
 \{ {B}(\omega_1, \omega_2) :
 {B}(\omega_1, \omega_2) \mbox{ is a  nonempty subset of }
 \caln_\varepsilon ( D (\omega_1, \omega_2) ),  \forall \
 \omega_1 \in \Omega_1,  \forall\  \omega_2
\in  \Omega_2\}
\ee
also belongs to $\cald$.
\end{defn}

\begin{defn} 
\label{temset} 
Let
$D=\{D(\omega_1, \omega_2): \omega_1 \in \Omega_1, \omega_2 \in \Omega_2 \}$
be a family of    nonempty subsets of $X$. 
We say $D$ is tempered in $X$ 
with respect to $(\Omega_1,  \{\thonet\}_{t \in \R})$
and
$(\Omega_2, \calftwo, P,  \{\thtwot\}_{t \in \R})$
if  there exists $x_0 \in X$ such that for every  $c>0$,
$\omega_1 \in \Omega_1$ and $\omega_2 \in \Omega_2$,
$$
\lim_{t \to -\infty}
e^{c t} d (x_0, D(\thonet \omega_1, \thtwot \omega_2))
=0.
$$
\end{defn}

\begin{defn} 
\label{defTlation}
Suppose   $T  \in \R$     and
  $\cald$  is  a collection  of some families of nonempty subsets of $X$
  as given by \eqref{defcald}.
  For  every  $D=\{D(\omega_1, \omega_2):
\omega_1 \in \Omega_1, \omega_2 \in \Omega_2 \} \in \cald$,  we write 
$$
D_T = \{ D_T(\omega_1, \omega_2): \  \ 
 D_T(\omega_1, \omega_2) = D(\theta_{1, T} \omega_1,  \omega_2), \ 
 \omega_1 \in \Omega_1, \ 
\omega_2 \in \Omega_2
\}.
$$
The family  $D_T$  is called    the  $T$-translation of
$D$. 
  Let $\cald_T$ be the collection of
    $T$-translations of  all elements of $\cald$, that is,
  $$
  \cald_T = \{ D_T: D_T \mbox{ is the }  T   \mbox{-translation of } D,
  \  D \in \cald \}.
  $$
  Then $\cald_T$ is called the $T$-translation of 
  the collection $\cald$.
  If $\cald_T \subseteq \cald$,   we say $\cald$ is
      $T$-translation   closed.
      If  $\cald_T =  \cald$,   we say $\cald$ is
      $T$-translation  invariant.
\end{defn}

One can check    that
  $\cald$ is $T$-translation  invariant
  if    and only if $\cald$ is both $-T$-translation closed   and $T$-translation 
  closed.  
 For later purpose, we need the concept of a complete orbit
 of $\Phi$ which is given below.

\begin{defn}
\label{comporbit}
 Let $\cald$ be a collection of some families of
 nonempty  subsets of $X$. A mapping $\psi: \R \times \Omega_1 \times \Omega_2$
 $\to X$ is called a complete orbit of $\Phi$ if for every $\tau \in \R$, $t \ge 0$,
 $\omega_1 \in \Omega_1$ and $\omega_2 \in \Omega_2$,  the following holds:
\be
\label{comporbit1}
 \Phi (t, \theta_{1, \tau} \omega_1, \theta_{2, \tau} \omega_2,
  \psi (\tau, \omega_1, \omega_2) )
  = \psi (t + \tau, \omega_1, \omega_2 ).
\ee
 If, in  addition,    there exists $D=\{D(\omega_1, \omega_2): \omega_1 \in \Omega,
 \omega_2 \in \Omega_2 \}\in \cald$ such that
 $\psi(t, \omega_1, \omega_2)$ belongs to
 $D(\theta_{1,t} \omega_1, \theta_{2, t} \omega_2 )$
 for every  $t \in \R$, $\omega_1 \in \Omega_1$
 and $\omega_2 \in \Omega_2$, then $\psi$ is called a
 $\cald$-complete orbit of $\Phi$.
 \end{defn}

\begin{defn}
\label{defomlit}
Let $B=\{B(\omega_1, \omega_2): \omega_1 \in \Omega_1, \ \omega_2  \in \Omega_2\}$
be a family of nonempty subsets of $X$.
For every $\omega_1 \in \Omega_1$ and
$\omega_2 \in \Omega_2$,  let
\be\label{omegalimit}
\Omega (B, \omega_1, \omega_2)
= \bigcap_{\tau \ge 0}
\  \overline{ \bigcup_{t\ge \tau} \Phi(t, \theta_{1,-t} \omega_1, \theta_{2, -t} \omega_2, B(\theta_{1,-t} \omega_1, \theta_{2,-t}\omega_2  ))}.
\ee
Then
the  family
 $\{\Omega (B, \omega_1, \omega_2): \omega_1 \in \Omega_1, \omega_2 \in \Omega_2 \}$
 is called the $\Omega$-limit set of $B$
 and is denoted by $\Omega(B)$.
 \end{defn}

\begin{defn}
Let $\cald$ be a collection of some families of nonempty subsets of $X$ and
$K=\{K(\omega_1, \omega_2): \omega_1 \in \Omega_1, \ \omega_2  \in \Omega_2\} \in \mathcal{D}$. Then
$K$  is called a  $\cald$-pullback
 absorbing
set for   $\Phi$   if
for all $\omega_1 \in \Omega_1$,
$\omega_2 \in \Omega_2 $
and  for every $B \in \cald$,
 there exists $T= T(B, \omega_1, \omega_2)>0$ such
that
\be
\label{abs1}
\Phi(t, \theta_{1,-t} \omega_1, \theta_{2, -t} \omega_2, B(\theta_{1,-t} \omega_1, \theta_{2,-t} \omega_2  ))  \subseteq  K(\omega_1, \omega_2)
\quad \mbox{for all} \ t \ge T.
\ee
If, in addition, for all $\omega_1 \in \Omega_1$ and $\omega_2 \in \Omega_2$,
   $K(\omega_1, \omega_2)$ is a closed nonempty subset of $X$
   and $K$ is measurable with respect to the $P$-completion of $\calftwo$
   in $\Omega_2$,
 then we say $K$ is a  closed measurable
  $\cald$-pullback absorbing  set for $\Phi$.
\end{defn}

\begin{defn}
\label{asycomp}
 Let $\cald$ be a collection of  some families of  nonempty
 subsets of $X$.
 Then
$\Phi$ is said to be  $\cald$-pullback asymptotically
compact in $X$ if
for all $\omega_1 \in \Omega_1$ and
$\omega_2 \in \Omega_2$,    the sequence
\be
\label{asycomp1}
\{\Phi(t_n, \theta_{1, -t_n} \omega_1, \theta_{2, -t_n} \omega_2,
x_n)\}_{n=1}^\infty \mbox{  has a convergent  subsequence  in }   X
\ee
 whenever
  $t_n \to \infty$, and $ x_n\in   B(\theta_{1, -t_n}\omega_1,
  \theta_{2, -t_n} \omega_2 )$   with
$\{B(\omega_1, \omega_2): \omega_1 \in \Omega_1, \ \omega_2 \in \Omega_2
\}   \in \mathcal{D}$.
\end{defn}

\begin{defn}
\label{defatt}
 Let $\cald$ be a collection of some families of
 nonempty  subsets of $X$
 and
 $\cala = \{\cala (\omega_1, \omega_2): \omega_1 \in \Omega_1,
  \omega_2 \in \Omega_2 \} \in \cald $.
Then     $\cala$
is called a    $\cald$-pullback    attractor  for
  $\Phi$
if the following  conditions (i)-(iii) are  fulfilled:
\begin{itemize}
\item [(i)]   $\cala$ is measurable
with respect to the $P$-completion of $\calftwo$ in $\Omega_2$ and
 $\cala(\omega_1, \omega_2)$ is compact for all $\omega_1 \in \Omega_1$
and    $\omega_2 \in \Omega_2$.

\item[(ii)]   $\cala$  is invariant, that is,
for every $\omega_1 \in \Omega_1$ and
 $\omega_2 \in \Omega_2$,
$$ \Phi(t, \omega_1, \omega_2, \cala(\omega_1, \omega_2)   )
= \cala (\theta_{1,t} \omega_1, \theta_{2,t} \omega_2
), \ \  \forall \   t \ge 0.
$$

\item[(iii)]   $\cala  $
attracts  every  member   of   $\cald$,  that is, for every
 $B = \{B(\omega_1, \omega_2): \omega_1 \in \Omega_1, \omega_2 \in \Omega_2\}
 \in \cald$ and for every $\omega_1 \in \Omega_1$ and
 $\omega_2 \in \Omega_2$,
$$ \lim_{t \to  \infty} d (\Phi(t, \theta_{1,-t}\omega_1, \theta_{2,-t}\omega_2, B(\theta_{1,-t}\omega_1, \theta_{2,-t}\omega_2) ) , \cala (\omega_1, \omega_2 ))=0.
$$
 \end{itemize}
 If, in addition, there exists $T>0$ such that
 $$
 \cala(\theta_{1, T} \omega_1, \omega_2) = \cala(\omega_1,    \omega_2 ),
 \quad \forall \  \omega_1 \in \Omega_1, \forall \
  \omega_2 \in \Omega_2,
 $$
 then we say $\cala$ is periodic with period $T$.
\end{defn}

The following result on the  
 existence and uniqueness of 
 $\cald$-pullback attractors 
for $\Phi$ can be found in \cite{wan4}. 
The reader is  referred to 
\cite{bat1,  cra2, fla1, schm1} for similar results
for random dynamical   systems.

 \begin{prop}
\label{att}  
 Let $\cald$ be a   neighborhood closed  
 collection of some  families of   nonempty subsets of
$X$,  and $\Phi$  be a continuous   cocycle on $X$
over $(\Omega_1,  \{\thonet\}_{t \in \R})$
and
$(\Omega_2, \calftwo, P,  \{\thtwot\}_{t \in \R})$.
Then
$\Phi$ has a  $\cald$-pullback
attractor $\cala$  in $\cald$
if and only if
$\Phi$ is $\cald$-pullback asymptotically
compact in $X$ and $\Phi$ has a  closed
   measurable (w.r.t. the $P$-completion of $\calftwo$)
     $\cald$-pullback absorbing set
  $K$ in $\cald$.
  The $\cald$-pullback
attractor $\cala$   is unique   and is given  by,
for each $\omega_1  \in \Omega_1$   and
$\omega_2 \in \Omega_2$,
\be\label{attform1}
\cala (\omega_1, \omega_2)
=\Omega(K, \omega_1, \omega_2)
=\bigcup_{B \in \cald} \Omega(B, \omega_1, \omega_2)
\ee
\be\label{attform2}
 =\{\psi(0, \omega_1, \omega_2): \psi \mbox{ is a   }  \cald {\rm -}
 \mbox{complete orbit of } \Phi\} .
 \ee
  \end{prop}

The periodicity of   $\cald$-pullback attractors  is 
proved in \cite{wan4} as given  below.

\begin{prop}
\label{periodatt}
Let   $T$   be   a positive number.
Suppose  $\Phi$   is  a continuous  periodic   cocycle
with period $T$  on $X$
over $(\Omega_1,  \{\thonet\}_{t \in \R})$
and
$(\Omega_2, \calftwo, P,  \{\thtwot\}_{t \in \R})$.
 Let   $\cald$   be  a  neighborhood closed 
   and $T$-translation invariant collection of
    some  families of   nonempty subsets of
$X$.
 If
$\Phi$ is $\cald$-pullback asymptotically
compact in $X$ and $\Phi$ has a  closed
   measurable (w.r.t. the $P$-completion of $\calftwo$)
     $\cald$-pullback absorbing set
  $K$ in $\cald$, then $\Phi$
  has a unique periodic
   $\cald$-pullback
attractor $\cala \in \cald$    with period $T$,  i.e., 
$\cala (\theta_{1, T} \omega_1,  \omega_2)
=\cala(\omega_1, \omega_2)$.
\end{prop}

\section{Cocycles  for Navier-Stokes Equations on Unbounded Domains}
\setcounter{equation}{0}

This section is devoted  to the existence of   
a continuous cocycle for the stochastic Navier-Stokes
equations with  non-autonomous deterministic terms.
Suppose   $Q$  is  an unbounded open set in $\R^2$ with boundary
$\partial Q$.
 Then consider       the following    stochastic 
   equations with multiplicative noise
    defined on $Q \times (\tau, \infty)$
   with  $\tau \in \R$:  
 \be
 \label{nse1}
 {\frac {\partial u}{\partial t}}
 -\nu \Delta u  + (u \cdot \nabla) u
 = f(x,t) -\nabla p
 +  \alpha u \circ {\frac {d w}{dt}}, \quad x \in Q \ \ \mbox{ and } \ t >\tau, 
 \ee
 \be\label{nse2}
 {\rm  div }\  u =0,
  \quad x \in Q \ \ \mbox{ and } \ t >\tau, 
 \ee
 with boundary condition
 \be\label{nse3}
 u = 0,   \quad x \in  \partial Q \ \ \mbox{ and } \ t >\tau, 
 \ee
 and initial condition
 \be
 \label{nse4}
 u(x, \tau) =u_\tau (x),  \quad x \in Q , 
 \ee
 where $\nu$ and $\alpha$
 are constants, $\nu>0$,  $f$ is a 
 given function defined on $Q \times \R$, 
 and $w$  
 is  a  two-sided real valued Wiener process 
 defined in a 
 probability space. 
 Note that equation \eqref{nse1} must be
 understood in the sense of Stratonovich  integration. 
 
 To reformulate problem \eqref{nse1}-\eqref{nse4},  we
 recall  the  
   standard function space:
   $$
   {\mathcal{V}} = \{ u \in C_0^\infty (Q) \times C_0^\infty (Q):
   {\rm div }\   u =0 \}.
   $$
   Let 
    $H $ 
    and  $V$   be  the closures of ${\mathcal{V}}$ in
    $L^2(Q) \times L^2(Q)$ and $H^1_0(Q) \times H^1_0(Q)$,
    respectively.     
    The dual space  of $V$ is denoted by $V^*$ with 
    norm $\| \cdot \|_{V^*}$.
    The duality pair between $V$   and
    $V^*$ is denoted by
    $\langle\cdot, \cdot \rangle$.
      Given $u, v \in V$, we    set
    $$
    (Du, Dv) = \sum_{i, j=1}^2 \int_Q 
    {\frac {\partial u_i}{\partial x_j}} 
 {\frac {\partial v_i}{\partial x_j}}  dx
 \quad {\rm and } \quad
 \| Du \| = (Du, Du)^{\frac 12}.
 $$
 For convenience, we write, for each
 $u, v, w \in V$,
 $$
 b(u,v,w) =  \sum_{i, j=1}^2 \int_Q 
 u_i {\frac {\partial u_j}{\partial x_i}} w_j dx .
 $$

Let $\{\thonet\}_{t \in \R}$  be a family 
of shift operators on $\R$   which is given by,
for each $t \in \R$,
 \be
 \label{shiftr}
 \thonet  (\tau) =   \tau + t, \quad \mbox{ for all } \ \tau \in  \R.
 \ee
 For  the  probability space we will use later, we write 
 $$   
\Omega = \{ \omega   \in C(\R, \R ): \ \omega(0) =  0 \}.
$$
Let $\calf$  be
 the Borel $\sigma$-algebra induced by the
compact-open topology of $\Omega$, and $P$
be  the corresponding Wiener
measure on $(\Omega, \calf)$. 
As usual, for each $t \in \R$  and
$\omega \in \Omega$, we may write
$w_t (\omega) = \omega (t)$.  
Denote by    $\{\thtwot \}_{t \in \R}$  
the standard group  on  
 $(\Omega, \calf, P)$:
\be\label{shiftome}
 \thtwot \omega (\cdot) = \omega (\cdot +t) - \omega (t), \quad  \omega \in \Omega, \ \ t \in \R .
\ee
Then $(\Omega, \mathcal{F}, P, \{\thtwot\}_{t\in \R})$ is a  parametric 
dynamical  system. In addition,        
  there exists a $\thtwot$-invariant set 
  $\tilde{\Omega}\subseteq \Omega$
of full $P$ measure  such that
for each $\omega \in \tilde{\Omega}$, 
\be\label{aspomega}
{\frac  {\omega (t)}{t}} \to 0 \quad \mbox {as } \ t \to \pm \infty.
\ee
From now on, we only consider the space 
$\tilde{\Omega}$ instead of $\Omega$, and hence
we  will write  $\tilde{\Omega}$ as 
$\Omega$  for   convenience.

  We next  define a continuous
   cocycle   for
   problem \eqref{nse1}-\eqref{nse4}
     in $H$  over
  $(\R, \{\thonet\}_{t \in \R})$  
  and $(\Omega, \calf, P, \{\thtwot\}_{t \in \R})$.
  To  this  end,   we need to
  transfer the stochastic equation into a deterministic one with
  random parameters.
  Given $t \in \R$  and $\omega \in \Omega$,   let
  $ z(t, \omega) = e^{-\alpha \omega  (t)}$. 
  Then   we find that  $z$ is a solution of  the equation
  \be\label{zequ1}
  d z = - \alpha z \circ d w.
 \ee
 Let $v$ be a new variable given by
 \be\label{vu}
 v(t, \tau, \omega, v_\tau) = z(t, \omega)
 u(t, \tau, \omega, u_\tau)
 \quad \mbox{ with } \ v_\tau = z(\tau, \omega) u_\tau.
\ee
 Formally,  from \eqref{nse1}-\eqref{nse4} and \eqref{zequ1} 
 we get that
  \be
 \label{v1}
 {\frac {\partial v}{\partial t}}
 -\nu \Delta v  + {\frac 1{z(t,\omega)}}  (v \cdot \nabla) v
 = z(t, \omega) \left (f(x,t) -\nabla p \right ), 
 \quad x \in Q \ \ \mbox{ and } \ t >\tau, 
 \ee
 \be\label{v2}
 {\rm  div }\  v =0,
  \quad x \in Q \ \ \mbox{ and } \ t >\tau, 
 \ee
 with boundary condition
 \be\label{v3}
 v = 0,   \quad x \in  \partial Q \ \ \mbox{ and } \ t >\tau, 
 \ee
 and initial condition
 \be
 \label{v4}
 v(x, \tau) =v_\tau (x),  \quad x \in Q .
 \ee

 Let  $\tau \in \R$, $\omega \in \Omega$,
  and
  $v_\tau \in H$. 
   A mapping $v (\cdot, \tau, \omega, v_\tau)$: $[\tau, \infty)
     \to H$  is called   a solution of problem
     \eqref{v1}-\eqref{v4} if  for every
   $T>0$,  
   $$v(\cdot, \tau, \omega, u_\tau)
     \in C([\tau, \infty), H)  \bigcap
   L^2 ((0,T), V)
 $$
and $v$  satisfies
\be\label{solv1}
(v(t), \zeta)
+ \nu \int_\tau^t (D v,  D \zeta) ds
+ \int_\tau^t {\frac 1{z(s, \omega)}}
b(v , v, \zeta) ds
=  (v_\tau, \zeta) + \int_\tau^t 
   z(s, \omega)  \langle    f(\cdot, s),  
  \  \zeta \rangle  ds,
\ee
for every $t \ge \tau$   and
 $\zeta \in V$.
 If,   in addition, 
 $v$ is 
 $(\calf, \calb(H))$-measurable
 with respect to $\omega \in \Omega$,
  we say   $v$ is a measurable solution of
  problem
     \eqref{v1}-\eqref{v4}.
     Since    \eqref{v1} is  a deterministic equation, it follows from 
     \cite{tem1}  that  for every $\tau \in \R$,
      $v_\tau \in H$   and $\omega \in \Omega$, problem 
      \eqref{v1}-\eqref{v4} has a unique solution
      $v$  in the sense
      of \eqref{solv1}
      which  continuously  depends  on $v_\tau$
      with the respect to the norm of $H$. Moreover, 
       the solution  $v$  is $(\calf, \calb (H))$-measurable
      in $\omega \in \Omega$.    
      This enables us to     define a cocycle
       $\Phi: \R^+ \times \R \times \Omega \times H$
$\to H$ for      problem
\eqref{nse1}-\eqref{nse4} by using \eqref{vu}.
Given $t \in \R^+$,  $\tau \in \R$, $\omega \in \Omega$ 
 and $u_\tau \in H$,
let 
 \be \label{nsephi}
 \Phi (t, \tau,  \omega, u_\tau) = 
  u (t+\tau,  \tau, \theta_{2, -\tau} \omega, u_\tau) 
  = {\frac 1{z(t+\tau, \theta_{2, -\tau} \omega)}}
v(t+\tau, \tau,  \theta_{2, -\tau} \omega,  v_\tau),  
\ee
where $v_\tau = z(\tau, \theta_{2, -\tau} \omega) u_\tau $. 
By \eqref{nsephi} we have,   for every
$t \ge 0$, $\tau \ge 0$, $r \in \R$,  
$\omega \in \Omega$ and $u_0 \in H$,
\be
\label{pcon1}
\Phi (t + \tau, r, \omega,  u_0)
=  {\frac 1{z(t+\tau +r, \theta_{2, -r} \omega)}}
v(t+\tau +r,  r,   \theta_{2, -r} \omega,  v_0), 
\ee
where $v_0 = z(r, \theta_{2,-r} \omega ) u_0$.
Similarly,  we have
$$
\Phi \left (t , \tau + r,  \theta_{2,\tau}\omega, 
 \Phi (\tau, r, \omega, u_0) 
\right )
$$
$$
=  {\frac 1{z(t+\tau +r, \theta_{2, -r} \omega)}}
v(t+\tau +r,  \tau +r,   \theta_{2, -r} \omega, 
 z(\tau +r, \theta_{2, -r} \omega)  \Phi (\tau, r, \omega, u_0)  )
$$
$$
=  {\frac 1{z(t+\tau +r, \theta_{2, -r} \omega)}}
v(t+\tau +r,  \tau +r,   \theta_{2, -r} \omega, 
  v(\tau +r, r,  \theta_{2, -r} \omega,  v_0) 
$$
\be
\label{pcon2}
=  {\frac 1{z(t+\tau +r, \theta_{2, -r} \omega)}}
v(t+\tau +r,  r,   \theta_{2, -r} \omega,  v_0).
\ee
It follows from \eqref{pcon1}-\eqref{pcon2} that
\be
\label{pcon3}
\Phi (t + \tau, r, \omega,  u_0)
=\Phi \left (t , \tau + r,  \theta_{2,\tau}\omega, 
 \Phi (\tau, r, \omega, u_0) 
\right ).
\ee
Since $v$  is  the   measurable solution of problem 
\eqref{v1}-\eqref{v4} which is continuous in initial data in $H$,
we find   from \eqref{pcon3}  that  $\Phi$ 
is a continuous   cocycle on $H$ over $(\R,  \{\thonet\}_{t \in \R})$
and 
$(\Omega, \calf,
P, \{\thtwot \}_{t\in \R})$.
The rest of this paper is devoted to the 
existence of pullback attractors
for $\Phi$   in $H$.  To  this end, we assume that   the open set
$Q$ is a Poincare domain in the sense    that   there
exists a positive number $\lambda$ such that
\be
\label{poincare}
\int_Q  |\nabla \phi (x) |^2  dx
\ge \lambda \int_Q |\phi  (x) |^2 dx,
\quad \mbox{   for   all } \ \phi \in H^1_0 (Q).
\ee

 Given a  bounded nonempty  subset 
 $B$  of $H$,  we write 
   $  \| B\| = \sup\limits_{\phi \in B}
   \| \phi\|_{H }$. 
Suppose 
   $D =\{ D(\tau, \omega): \tau \in \R, \omega \in \Omega \}$  
    is   a  tempered family of
  bounded nonempty   subsets of $H $,  that is,  
  for every  $c>0$, $\tau \in \R$   and $\omega \in \Omega$, 
 \be
 \label{attdom1}
 \lim_{r \to   \infty} e^{ -  c  r} 
 \| D( \tau  -r, \theta_{2, -r} \omega ) \|  =0. 
 \ee 
Let $\cald$  be      the  collection of all  tempered families of
bounded nonempty  subsets of $H$, i.e.,
 \be
 \label{dnse}
\cald = \{ 
   D =\{ D(\tau, \omega): \tau \in \R, \omega \in \Omega \}: \ 
 D  \ \mbox{satisfies} \  \eqref{attdom1} \} .
\ee
From \eqref{dnse}   we see  that
$\cald$ is  neighborhood closed.
For later purpose,   we assume that   the external term
$f$ satisfies    the  following condition:   there exists
a number $\delta \in [0, \nu\lambda)$
such that 
 \be
 \label{fcond1}
 \int_{-\infty}^\tau e^{\delta  r } \| f(\cdot, r)\|^2_ {V^*}d r
<  \infty, \quad \forall \ \tau \in \R.
 \ee
 When proving the existence of tempered pullback absorbing
 sets for the Navier-Stokes equations, we also assume  that
 there   exists
  $\delta \in [0, \nu\lambda)$
such that  for every positive number $c$, 
 \be
 \label{fcond2}
 \lim_{r \to -\infty} e^{c r} 
 \int_{-\infty}^0   e^{\delta  s}
  \|f(\cdot,  s+r) \|^2_{V^*}    d s  =0.
  \ee
  Note that \eqref{fcond2} implies \eqref{fcond1} if
  $f \in L^2_{loc} (\R, V^*)$. 
  It is worth pointing out that
 both  conditions  \eqref{fcond1} 
 and \eqref{fcond2} 
do   not require that    $f$ 
is  bounded in $V^*$
at $ \pm \infty$.  For instance,  for any $\beta \ge 0$  and
$f_1 \in V^*$,  the function $ f(\cdot, t) = t^\beta f_1  $
satisfies both     \eqref{fcond1} 
 and \eqref{fcond2}.

\section{Uniform Estimates of Solutions}
\setcounter{equation}{0}

  In this section,  we
 derive uniform estimates  on the 
  solutions  of  problem \eqref{v1}-\eqref{v4} 
  and then prove the $\cald$-pullback asymptotic compactness
  of the solutions    by the    idea  of energy equations
   as introduced by Ball in \cite{bal1}
   for deterministic systems.

\begin{lem}
\label{lem1}
 Suppose  \eqref{poincare} and \eqref{fcond1} hold.
Then for every $\tau \in \R$, $\omega \in \Omega$   and $D=\{D(\tau, \omega)
: \tau \in \R,  \omega \in \Omega\}  \in \cald$,
 there exists  $T=T(\tau, \omega,  D)>0$ such that for all $t \ge T$ and
 $s  \ge  \tau -t $, the solution
 $v$ of  problem  \eqref{v1}-\eqref{v4}  with $\omega$ replaced by
 $\theta_{2, -\tau} \omega$  satisfies
$$
\| v(s , \tau -t,  \theta_{2, -\tau} \omega, v_{\tau -t}  ) \|^2 
 \le 
 e^{\nu \lambda (\tau -s)}
 + {\frac 2\nu} e^{-\nu \lambda s}
 \int_{-\infty}^s
 e^{\nu\lambda r}
 z^2(r, \theta_{2,-\tau} \omega) \| f(\cdot, r)\|_{V^*}^2 dr,
$$
and
 $$
  \int_{\tau -t}^s e^{ \nu \lambda r}
  \| Dv(r, \tau -t,\theta_{2, -\tau} \omega, v_{\tau -t} ) \|^2 dr
\le {\frac 2\nu} 
e^{\nu \lambda \tau  }
 +  {\frac 4{\nu^2}}  
 \int_{-\infty}^s
 e^{\nu\lambda r}
 z^2(r, \theta_{2,-\tau} \omega) \| f(\cdot, r)\|_{V^*}^2 dr,
$$
 where $v_{\tau -t}\in D(\tau -t, \theta_{2, -t} \omega)$.
\end{lem}

\begin{proof}
 Formally, it follows   from \eqref{v1}-\eqref{v3}  that for
each $\tau \in \R$, $t \ge 0$  and $\omega \in \Omega$, 
\be\label{plem1_1}
{\frac 12} {\frac d{dt}} \| v\|^2  +
\nu \| Dv \|^2
= z(t, \omega) \langle f(\cdot, t), v \rangle .
\ee
The right-hand side of \eqref{plem1_1} is bounded by
$$|z(t, \omega) \langle f(\cdot, t), v \rangle |
\le {\frac 14}\nu \| D v\|^2
+ {\frac 1\nu} z^2(t,  \omega ) \| f(\cdot, t)\|^2_{V^*}.
$$
Therefore,  from \eqref{plem1_1}   we get
\be\label{plem1_2}
 {\frac d{dt}} \| v\|^2  +
{\frac 32} \nu \| Dv \|^2
\le
{\frac 2\nu} z^2(t,  \omega ) \| f(\cdot, t)\|^2_{V^*}.
\ee
By \eqref{poincare} and \eqref{plem1_2}   we have
\be\label{plem1_3}
 {\frac d{dt}} \| v\|^2  
 +\nu\lambda \| v \|^2 +
{\frac 12} \nu \| Dv \|^2
\le
{\frac 2\nu} z^2(t,  \omega ) \| f(\cdot, t)\|^2_{V^*}.
\ee
Multiplying \eqref{plem1_3}   by $e^{\nu\lambda t}$ and then
integrating the inequality on $[\tau -t, s]$,
we obtain
$$
\| v(s, \tau -t, \omega, v_{\tau -t} )\|^2
+{\frac 12} \nu
\int_{\tau -t}^s
e^{\nu \lambda (r-s)} 
\|D v(r, \tau -t, \omega, v_{\tau-t}) \|^2 dr
$$
$$
\le
e^{\nu\lambda (\tau -s)} e^{-\nu \lambda t} \| v_{\tau -t} \|^2
+{\frac 2\nu} \int_{\tau -t} ^s
e^{\nu\lambda (r-s)} z^2(r, \omega) \| f(\cdot, r)\|^2_{V^*} dr.
$$
Replacing $\omega$   by 
$\theta_{2, -\tau} \omega$ in the above,   we get that
$$
\| v(s, \tau -t, \theta_{2, -\tau} \omega, v_{\tau -t} )\|^2
+{\frac 12} \nu
\int_{\tau -t}^s
e^{\nu \lambda (r-s)} 
\|D v(r, \tau -t,  \theta_{2, -\tau} \omega, v_{\tau-t}) \|^2 dr
$$
\be\label{plem1_5}
\le
e^{\nu\lambda (\tau -s)} e^{-\nu \lambda t} \| v_{\tau -t} \|^2
+{\frac 2\nu} e^{- \nu\lambda  s }
\int_{\tau -t} ^s
e^{\nu\lambda r} z^2(r,  \theta_{2, -\tau} \omega) \| f(\cdot, r)\|^2_{V^*} dr.
\ee
We now  estimate    the last   term
on the right-hand side  of \eqref{plem1_5}.
Let $\tomega = \theta_{2, -\tau}  \omega$. Then 
by \eqref{aspomega} we find that there  exists $R<0$ such that
for all $r \le R$,
$$
-2\alpha \tomega (r)
\le -(\nu \lambda -\delta) r,
$$
where $\delta$ is the positive constant
in \eqref{fcond1}. Therefore, for   all $r \le R$,
\be
\label{plem1_6}
z^2(r, \tomega)
=e^{-2\alpha \tomega (r)}
 \le e^{-(\nu\lambda -\delta)r}.
\ee
By \eqref{plem1_6}   we have for all $r \le R$,
$$
e^{\nu\lambda r} z^2(r,  \theta_{2, -\tau} \omega) \| f(\cdot, r)\|^2_{V^*} 
=
e^{(\nu\lambda-\delta) r} 
z^2(r,  \tomega)  e^{\delta r} \| f(\cdot, r)\|^2_{V^*} 
\le e^{\delta r} \| f(\cdot, r)\|^2_{V^*} ,
$$
which along with \eqref{fcond1}  shows    that
for every $s \in \R$, $\tau \in \R$    and   
$\omega \in \Omega$,
\be\label{plem1_7}
\int_{-\infty}^s
e^{\nu\lambda r} z^2(r,  \theta_{2, -\tau} \omega) \| f(\cdot, r)\|^2_{V^*}  dr
< \infty.
\ee
On the other hand,
since $v_{\tau -t} \in D(\tau -t, \theta_{2, -t} \omega)$,
  for the first term on the
right-hand side of \eqref{plem1_5}, we have
$$
e^{-\nu \lambda t} \| v_{\tau -t} \|^2
\le 
e^{-\nu \lambda t} \| D(\tau -t, \theta_{2, -t} \omega ) \|^2
\to 0,\quad
\mbox{ as  }  \  t \to  \infty.$$
This shows  that there exists
$T= T(\tau, \omega, D)>0$ such that
$e^{-\nu \lambda t} \| v_{\tau -t} \|^2 \le 1$
for   all $t \ge T$.
Thus, the   first term on the
right-hand side of \eqref{plem1_5} satisfies
\be\label{plem1_8}
e^{\nu\lambda (\tau -s)} e^{-\nu \lambda t} \| v_{\tau -t} \|^2
\le
e^{\nu\lambda (\tau -s)},
\quad \mbox{ for all } \  t\ge T.
\ee
From  \eqref{plem1_5}, \eqref{plem1_7} and \eqref{plem1_8}, 
the lemma    follows.   \end{proof}

As an immediate consequence of Lemma \ref{lem1}, we have the following estimates
on the solutions of problem \eqref{v1}-\eqref{v4}.

\begin{lem}
\label{lem2}
 Suppose  \eqref{poincare} and \eqref{fcond1} hold.
Then for every $\tau \in \R$, $\omega \in \Omega$   and $D=\{D(\tau, \omega)
: \tau \in \R,  \omega \in \Omega\}  \in \cald$,
 there exists  $T=T(\tau, \omega,  D)>0$ such that for 
 every $k \ge 0$   and    for all $t \ge T +k$, 
 the solution
 $v$ of  problem  \eqref{v1}-\eqref{v4}  with $\omega$ replaced by
 $\theta_{2, -\tau} \omega$  satisfies
$$
\| v(\tau -k , \tau -t,  \theta_{2, -\tau} \omega, v_{\tau -t}  ) \|^2 
 \le 
 e^{\nu \lambda k}
 + {\frac 2\nu} e^{\nu \lambda (k -\tau)}
 \int_{-\infty}^{\tau -k}
 e^{\nu\lambda r}
 z^2(r, \theta_{2,-\tau} \omega) \| f(\cdot, r)\|_{V^*}^2 dr,
 $$
 where $v_{\tau -t}\in D(\tau -t, \theta_{2, -t} \omega)$.
\end{lem}

\begin{proof}
Given $\tau \in \R$    and $k \ge 0$, let
$s= \tau -k$.   Let $T= T(\tau, \omega, D)$ be   the positive constant
claimed in Lemma \ref{lem1}.  If 
  $t \ge T + k$,  then   we have
  $ t\ge T$   and $ s \ge \tau -t$. 
Thus,  
the desired result follows   from Lemma \ref{lem1}. 
\end{proof}

 Next,  we prove  the $\cald$-pullback asymptotic compactness of the
 solutions  of problem \eqref{v1}-\eqref{v4}.  For this purpose, we need
 the following weak continuity of solutions in initial data, which can be established 
 by  the standard  methods  as in \cite{ros1}.
 
 \begin{lem}
 \label{lem3}
 Suppose  \eqref{poincare} holds 
 and $f \in L^2_{loc} (\R, V^*)$. Let
 $\tau \in \R$, $\omega \in \Omega$,
 $v_\tau $, $v_{\tau, n} \in  H$ for
 all $n \in \N$.
 If $v_{\tau, n}   \rightharpoonup v_\tau$ in $H$, 
  then  the solution $v$   of problem \eqref{v1}-\eqref{v4}
  has  the  properties:
  $$
  v(r, \tau, \omega, v_{\tau, n}) \rightharpoonup
  v(r, \tau, \omega, v_\tau) \quad \mbox{ in } \  H
  \ \mbox{ for   all } \  r \ge \tau,
  $$
  and
  $$
   v(\cdot,  \tau, \omega, v_{\tau, n}) \rightharpoonup
  v(\cdot, \tau, \omega, v_\tau) \quad \mbox{ in } \  L^2 ((\tau, \tau +T), V)
  \ \mbox{ for   every } \   T>0.
  $$
 \end{lem}
 
 The next lemma   is concerned with the
  pullback asymptotic compactness of
 problem \eqref{v1}-\eqref{v4}.
 
 \begin{lem}
 \label{lem4}
Suppose  \eqref{poincare} and \eqref{fcond1} hold.
Then for every $\tau \in \R$, $\omega \in \Omega$,  $D=\{D(\tau, \omega)
: \tau \in \R,  \omega \in \Omega\}  \in \cald$ and 
$t_n \to \infty$,
 $v_{0,n}  \in D(\tau -t_n, \theta_{2, -t_n} \omega )$,  the sequence
 $v(\tau, \tau -t_n,  \theta_{2, -\tau} \omega,   v_{0,n}  ) $ 
 of solutions of problem \eqref{v1}-\eqref{v4}
   has a
   convergent
subsequence in   $H$.
\end{lem}

\begin{proof}
It follows   from Lemma \ref{lem2} with $k =0$ that, there  exists
$T=T(\tau, \omega, D)>0$    such that   for   all $t \ge T$,
\be\label{plem4_1a1}
\| v(\tau, \tau -t, \theta_{2, -\tau} \omega, v_{\tau -t} ) \|^2
\le 1 
 + {\frac 2\nu} e^{- \nu \lambda \tau }
 \int_{-\infty}^{\tau }
 e^{\nu\lambda r}
 z^2(r, \theta_{2,-\tau} \omega) \| f(\cdot, r)\|_{V^*}^2 dr,
 \ee
 with $v_{\tau -t} \in D(\tau -t, \theta_{2, -t} \omega )$.
 Since   
 $t_n \to \infty$,  there  exists  $N_0 \in \N$ such that   
   $t _n \ge T$   for    all   $ n \ge N_0$. 
   Due to   $v_{0,n}  \in D(\tau -t_n, \theta_{2, -t_n} \omega )$,
   we get  from    \eqref{plem4_1a1}
   that   for all $n \ge   \N_0$,
 \be\label{plem4_1}
 \| v(\tau, \tau -t_n, \theta_{2, -\tau} \omega, v_{0, n} ) \|^2
\le 1 
 + {\frac 2\nu} e^{- \nu \lambda \tau }
 \int_{-\infty}^{\tau }
 e^{\nu\lambda r}
 z^2(r, \theta_{2,-\tau} \omega) \| f(\cdot, r)\|_{V^*}^2 dr.
 \ee
  By \eqref{plem4_1} there exists $\tilv \in H$  
  and a subsequence (which  is not relabeled) such that
  \be
  \label{plem4_2}
  v(\tau, \tau -t_n, \theta_{2, -\tau} \omega, v_{0, n} ) 
   \rightharpoonup \tilv \quad \mbox{  in } \ H.
   \ee
   We now prove that  the weak convergence  of \eqref{plem4_2}
   is actually   a  strong convergence,  which will complete
   the proof.
   Note    that \eqref{plem4_2}   implies    
   \be
   \label{plem4_3}
   \liminf_{n \to \infty}
    \|v(\tau, \tau -t_n, \theta_{2, -\tau} \omega, v_{0, n} ) \|
    \ge \| \tilv \|.
    \ee
    So we  only need   to show
     \be
   \label{plem4_4}
   \limsup_{n \to \infty}
    \|v(\tau, \tau -t_n, \theta_{2, -\tau} \omega, v_{0, n} ) \|
     \le  \| \tilv \|.
    \ee
    We  will establish   \eqref{plem4_4} by the method of  energy equations
    due to Ball \cite{bal1}.
    Given $k \in \N$   we have
    \be
    \label{plem4_5}
    v(\tau, \tau -t_n, \theta_{2, -\tau} \omega, v_{0, n} ) 
    = v(\tau, \tau - k, \theta_{2, -\tau} \omega, \  
        v(\tau-k, \tau -t_n, \theta_{2, -\tau} \omega, v_{0, n} )    ).
        \ee
        For each $k$, let $N_k$   be large enough such that
        $t_n \ge T+ k$   for all $ n \ge N_k$.
       Then it   follows   from Lemma  \ref{lem2}    that
        for  $n \ge N_k $,
         $$
\| v(\tau -k , \tau -t_n,  \theta_{2, -\tau} \omega, v_{0,n}  ) \|^2 
 \le 
 e^{\nu \lambda k}
 + {\frac 2\nu} e^{\nu \lambda (k -\tau)}
 \int_{-\infty}^{\tau -k}
 e^{\nu\lambda r}
 z^2(r, \theta_{2,-\tau} \omega) \| f(\cdot, r)\|_{V^*}^2 dr,
 $$
 which     shows   that, for each fixed $k \in \N$,   the sequence 
 $ v(\tau -k , \tau -t_n,  \theta_{2, -\tau} \omega, v_{0,n}  ) $
 is bounded in $H$. By a diagonal process,   one can  find 
 a  subsequence (which we do not relabel) 
and a point 
 $\tilv_k \in H$  for   each $k \in \N$      such that
 \be
 \label{plem4_6}
  v(\tau -k , \tau -t_n,  \theta_{2, -\tau} \omega, v_{0,n}  )
  \rightharpoonup \tilv_k
  \quad \mbox{ in } \ H.
  \ee
  By \eqref{plem4_5}-\eqref{plem4_6}   and Lemma \ref{lem3}
  we get that
  for  each $k \in \N$, 
  \be
  \label{plem4_10}
    v(\tau, \tau -t_n, \theta_{2, -\tau} \omega, v_{0, n} ) 
    \rightharpoonup
    v(\tau, \tau - k, \theta_{2, -\tau} \omega,  \tilv_k ) 
    \quad \mbox{ in } \ H,
    \ee
    and
    \be
    \label{plem4_11}
    v(\cdot, \tau - k, \theta_{2, -\tau} \omega, \  
        v(\tau-k, \tau -t_n, \theta_{2, -\tau} \omega, v_{0, n} )    )
        \rightharpoonup
        v(\cdot, \tau - k, \theta_{2, -\tau} \omega,   \tilv_k )
        \quad \mbox{in } \ L^2((\tau -k,  \tau), V) .
        \ee
By \eqref{plem4_2}  and \eqref{plem4_10}   we have
\be
\label{plem4_12}
v(\tau, \tau - k, \theta_{2, -\tau} \omega,  \tilv_k ) =\tilv.
\ee
Note that \eqref{plem1_1}  implies  that
  \be\label{plem4_20}
 {\frac d{dt}} \| v\|^2  +
\nu \lambda  \| v \|^2 + \psi (v)
= 2 z(t, \omega) \langle f(\cdot, t), v \rangle,
\ee  
where $\psi$   is a   functional on $V$  given   by 
$$ \psi (v)  = 2\nu \| Dv \|^2 - \nu \lambda \| v \|^2,
\quad \mbox{  for  all } \ v \in V.
$$
By \eqref{poincare}   we see   that
$$ \nu \| D v \|^2  \le \psi (v) \le 2 \nu \| Dv \|^2,
\quad \mbox{  for  all } \ v \in V.
$$
This indicates    that $\psi (\cdot) $ is an equivalent norm
of $V$.  It   follows   from \eqref{plem4_20} that
for each $\omega \in \Omega$, $s \in \R$  and $\tau \ge s$,
$$
\| v(\tau, s,  \omega, v_s) \|^2
=e^{\nu\lambda (s -\tau)} \| v_s \|^2
-\int_s^\tau e^{\nu \lambda (r-\tau)} \psi (v(r, s, \omega, v_s))  dr
$$
\be\label{plem4_21}
+2 \int_s^\tau  e^{\nu \lambda (r-\tau)} 
z(r, \omega) \langle f(\cdot, r), v(r,s,\omega, v_s )\rangle dr.
\ee
By  \eqref{plem4_12} and \eqref{plem4_21} we find that
$$
\| \tilv\|^2 = 
  \| v(\tau, \tau -k,  \theta_{2, -\tau} \omega,  \tilv_k) \|^2
=e^{ - \nu\lambda  k } \|  \tilv_k \|^2
- \int_{\tau -k}^\tau e^{\nu \lambda (r-\tau)}
 \psi (v(r, \tau -k, \theta_{2, -\tau} \omega,  \tilv_k ))  dr
 $$
\be\label{plem4_22}
+2 \int_{\tau -k}^\tau  e^{\nu \lambda (r-\tau)} 
z(r,  \theta_{2, -\tau} \omega)
 \langle f(\cdot, r), v(r, \tau -k,  \theta_{2, -\tau}\omega,  \tilv_k )\rangle dr.
\ee
    Similarly,  by  \eqref{plem4_5}
    and \eqref{plem4_21} we obtain  that 
  $$
    \| v(\tau, \tau -t_n, \theta_{2, -\tau} \omega, v_{0, n} )  \|^2
    = \|v(\tau, \tau - k, \theta_{2, -\tau} \omega, \  
        v(\tau-k, \tau -t_n, \theta_{2, -\tau} \omega, v_{0, n} )    ) \|^2
        $$
   $$      =e^{ - \nu\lambda  k } 
   \|  v(\tau-k, \tau -t_n, \theta_{2, -\tau} \omega, v_{0, n} )    \|^2
   $$
   $$
- \int_{\tau -k}^\tau e^{\nu \lambda (r-\tau)}
 \psi (v(r, \tau -k, \theta_{2, -\tau} \omega, 
  v(\tau-k, \tau -t_n, \theta_{2, -\tau} \omega, v_{0, n} ) ) )  dr
 $$
\be\label{plem4_23}
+2 \int_{\tau -k}^\tau  e^{\nu \lambda (r-\tau)} 
z(r,  \theta_{2, -\tau} \omega)
 \langle f(\cdot, r), 
 v(r, \tau -k,  \theta_{2, -\tau}\omega,  
  v(\tau-k, \tau -t_n, \theta_{2, -\tau} \omega, v_{0, n} )  )\rangle dr.
\ee
We now  consider the limit of each  term on the right-hand
side of \eqref{plem4_23}   as $n \to \infty$.
For the first  term, 
by \eqref{plem1_5} with $s = \tau -k$ and $t =t_n$ 
 we get  that
$$
e^{- \nu\lambda k}
\| v(\tau -k, \tau -t_n, \theta_{2, -\tau} \omega, v_{0,n} )\|^2
$$
 \be
 \label{plem4_24}
\le
e^{- \nu\lambda t_n }   \| v_{0,n} \|^2
+{\frac 2\nu} e^{- \nu\lambda  \tau }
\int_{-\infty} ^{\tau -k}
e^{\nu\lambda r} z^2(r,  \theta_{2, -\tau} \omega) \| f(\cdot, r)\|^2_{V^*} dr.
\ee
Since $v_{0,n} \in D(\tau -t_n, \theta_{2, -t_n} \omega)$ we have
$$
e^{- \nu\lambda t_n }   \| v_{0,n} \|^2
\le  e^{- \nu\lambda t_n }   \| D(\tau -t_n, \theta_{2, -t_n} \omega)  \|^2
\to 0 \quad \mbox{as } \ n \to \infty,
$$
which   along with \eqref{plem4_24} shows   that
$$
\limsup_{n \to \infty} e^{- \nu\lambda k}
\| v(\tau -k, \tau -t_n, \theta_{2, -\tau} \omega, v_{0,n} )\|^2
$$
\be
 \label{plem4_30}
\le
{\frac 2\nu} e^{- \nu\lambda  \tau }
\int_{-\infty} ^{\tau -k}
e^{\nu\lambda r} z^2(r,  \theta_{2, -\tau} \omega) \| f(\cdot, r)\|^2_{V^*} dr.
\ee
By \eqref{plem4_11} we find   that
$$
\lim_{n \to \infty}
\int_{\tau -k}^\tau  e^{\nu \lambda (r-\tau)} 
z(r,  \theta_{2, -\tau} \omega)
 \langle f(\cdot, r), 
 v(r, \tau -k,  \theta_{2, -\tau}\omega,  
  v(\tau-k, \tau -t_n, \theta_{2, -\tau} \omega, v_{0, n} )  )\rangle dr
  $$
  \be
  \label{plem4_40}
  =
  \int_{\tau -k}^\tau  e^{\nu \lambda (r-\tau)} 
z(r,  \theta_{2, -\tau} \omega)
 \langle f(\cdot, r), 
 v(r, \tau -k,  \theta_{2, -\tau}\omega,  \tilv_k ) \rangle dr,
 \ee
 and
 $$
 \liminf_{n \to \infty}
 \int_{\tau -k}^\tau e^{\nu \lambda (r-\tau)}
 \psi (v(r, \tau -k, \theta_{2, -\tau} \omega, 
  v(\tau-k, \tau -t_n, \theta_{2, -\tau} \omega, v_{0, n} ) ) )  dr
  $$
  \be\label{plem4_41}
  \ge
   \int_{\tau -k}^\tau e^{\nu \lambda (r-\tau)}
 \psi (v(r, \tau -k, \theta_{2, -\tau} \omega, 
   \tilv_k  ))   dr.
   \ee
   Note  that \eqref{plem4_41} implies   that
    $$
  \limsup _{n \to \infty} - 
 \int_{\tau -k}^\tau e^{\nu \lambda (r-\tau)}
 \psi (v(r, \tau -k, \theta_{2, -\tau} \omega, 
  v(\tau-k, \tau -t_n, \theta_{2, -\tau} \omega, v_{0, n} ) ) )  dr
  $$
  \be\label{plem4_42}
  \le -
   \int_{\tau -k}^\tau e^{\nu \lambda (r-\tau)}
 \psi (v(r, \tau -k, \theta_{2, -\tau} \omega, 
   \tilv_k  ))   dr.
   \ee
   Taking   the limit of \eqref{plem4_23} as $n \to \infty$, by
   \eqref{plem4_30}, \eqref{plem4_40} and \eqref{plem4_42} we
   obtain  that
   $$
   \limsup _{n \to \infty}
      \| v(\tau, \tau -t_n, \theta_{2, -\tau} \omega, v_{0, n} )  \|^2
      \le
{\frac 2\nu} e^{- \nu\lambda  \tau }
\int_{-\infty} ^{\tau -k}
e^{\nu\lambda r} z^2(r,  \theta_{2, -\tau} \omega) \| f(\cdot, r)\|^2_{V^*} dr
$$
$$
  - \int_{\tau -k}^\tau e^{\nu \lambda (r-\tau)}
 \psi (v(r, \tau -k, \theta_{2, -\tau} \omega,  \tilv_k ))  dr
 $$
\be\label{plem4_50}
+2 \int_{\tau -k}^\tau  e^{\nu \lambda (r-\tau)} 
z(r,  \theta_{2, -\tau} \omega)
 \langle f(\cdot, r), v(r, \tau -k,  \theta_{2, -\tau}\omega,  \tilv_k )\rangle dr.
\ee
It follows   from \eqref{plem4_22}  and \eqref{plem4_50}
that
 \be
 \label{plem4_60}
   \limsup _{n \to \infty}
      \| v(\tau, \tau -t_n, \theta_{2, -\tau} \omega, v_{0, n} )  \|^2
      \le \| \tilv \|^2    + 
{\frac 2\nu} e^{- \nu\lambda  \tau }
\int_{-\infty} ^{\tau -k}
e^{\nu\lambda r} z^2(r,  \theta_{2, -\tau} \omega) \| f(\cdot, r)\|^2_{V^*} dr.
\ee
Let $ k \to \infty$ in  \eqref{plem4_60} to yield
 \be
 \label{plem4_61}
   \limsup _{n \to \infty}
      \| v(\tau, \tau -t_n, \theta_{2, -\tau} \omega, v_{0, n} )  \|^2
      \le \| \tilv \|^2  .
      \ee
      By \eqref{plem4_2}-\eqref{plem4_3}   and \eqref{plem4_61}
      we find   that
      $$
      \lim _{n \to \infty}
       v(\tau, \tau -t_n, \theta_{2, -\tau} \omega, v_{0, n} )   
      =    \tilv  \quad \mbox {in } H .
      $$
      This completes    the proof.
\end{proof}

\section{Existence of  Pullback  Attractors }
 \setcounter{equation}{0}
 
 In   this  section,  we establish   the existence of  $\cald$-pullback attractors
 for the Navier-Stokes equations \eqref{nse1}-\eqref{nse2}. 
 Based on  the uniform
 estimates  on the   solutions   of problem \eqref{v1}-\eqref{v4},  we first show
 that the cocycle  $\Phi$    associated  with the  stochastic  system 
 \eqref{nse1}-\eqref{nse4}  has a  measurable $\cald$-pullback  absorbing   set
 in $H$, and then prove  the $\cald$-pullback   asymptotic compactness of $\Phi$.

\begin{lem}
\label{lematt1}
 Suppose  \eqref{poincare} and \eqref{fcond1} hold.
Then for every $\tau \in \R$, $\omega \in \Omega$   and $D=\{D(\tau, \omega)
: \tau \in \R,  \omega \in \Omega\}  \in \cald$,
 there exists  $T=T(\tau, \omega,  D)>0$ such that 
   for all $t \ge T $, 
 the solution
 $u$ of  problem  \eqref{nse1}-\eqref{nse4}  with $\omega$ replaced by
 $\theta_{2, -\tau} \omega$  satisfies
$$
\| u(\tau  , \tau -t,  \theta_{2, -\tau} \omega, u_{\tau -t}  ) \|^2 
$$
$$
 \le 
 z^{-2} (\tau, \theta_{2, -\tau} \omega )
 + {\frac 2\nu}z^{-2} (\tau, \theta_{2, -\tau} \omega )
 \int_{-\infty}^{\tau }
 e^{\nu\lambda (r-\tau)}
 z^2(r, \theta_{2,-\tau} \omega) \| f(\cdot, r)\|_{V^*}^2 dr,
 $$
 where $u_{\tau -t}\in D(\tau -t, \theta_{2, -t} \omega)$.
\end{lem}

\begin{proof}
Given $D=\{D(\tau, \omega): \tau \in \R, \omega \in \Omega\}
 \in \cald$,  for each $\tau \in \R$   and
 $\omega \in \Omega$,  denote by
 \be
 \label{plematt1_1}
{\tilde{D}}(\tau, \omega)
 = \{ v \in H :  \| v \|
 \le |z(\tau, \theta_{2, -\tau} \omega)| \ \| D(\tau, \omega) \| \}.
 \ee
 Let ${\tilde{D}}$  be a  family
 corresponding to $D$ which  consists of    
  the sets given by \eqref{plematt1_1}, i.e.,
 \be\label{plematt1_2}
 {\tilde{D}} = 
\{ {\tilde{D}}(\tau, \omega) : 
  {\tilde{D}}(\tau, \omega) \mbox{ is    defined by }
    \eqref{plematt1_1},  \tau \in \R,  \omega \in \Omega \}.
\ee
We now  prove  
${\tilde{D}}$  is tempered in $H$
 for  $D \in \cald$.
Given  $c>0$, by \eqref{aspomega} we find that
for  each $\omega \in \Omega$, 
there   exists $R>0$ such    that   for   all  
$r \ge R$,
\be
\label{plematt1_3}
| -\alpha \omega (-r) | \le {\frac 12} c r.
\ee
Since $D\in \cald$,  we get  from  \eqref{plematt1_3}  
that
$$
e^{-cr} \| {\tilde{D}} (\tau -r, \theta_{2,-r} \omega ) \|
= 
e^{-cr} |z(\tau -r, \theta_{2, -\tau} \omega )| \ 
\| D  (\tau -r, \theta_{2,-r} \omega ) \|
$$
$$
\le e^{\alpha \omega (-\tau) } 
e^{- {\frac 12} cr}  
\| D  (\tau -r, \theta_{2,-r} \omega ) \|
\to 0,
\quad \mbox{  as } \  r \to \infty,
$$
which  shows that ${\tilde{D}} \in \cald$. 
Since $u_{\tau -t}\in D(\tau -t, \theta_{2, -t} \omega)$, by 
\eqref{vu} we  know  that
$$ \| v_{\tau -t} \|
= \| z(\tau -t, \theta_{2, -\tau} \omega ) \ u_{\tau -t} \|
\le
|  z(\tau -t, \theta_{2, -\tau} \omega )| 
 \ \| D(\tau -t, \theta_{2, -t} \omega)\|,
$$
which along with \eqref{plematt1_1} implies  that
$v_{\tau -t} \in {\tilde{D}} (\tau -t, \theta_{2, -t} \omega )$.
Since ${\tilde{D}} $  is tempered, it follows   from Lemma \ref{lem2} with
$k=0$   that there  exists  $T= T(\tau, \omega, D)>0$   such that
for all $ t \ge T$,
$$
\| v(\tau  , \tau -t,  \theta_{2, -\tau} \omega, v_{\tau -t}  ) \|^2 
 \le  1
 + {\frac 2\nu} 
 \int_{-\infty}^{\tau }
 e^{\nu\lambda (r-\tau)}
 z^2(r, \theta_{2,-\tau} \omega) \| f(\cdot, r)\|_{V^*}^2 dr,
 $$
 which along with \eqref{vu} completes   the proof.
\end{proof}

 \begin{lem}
 \label{lematt2}
 Suppose  \eqref{poincare}  and   \eqref{fcond2}  hold.
 Then the continuous cocycle $\Phi$ associated with
 problem \eqref{nse1}-\eqref{nse4} has a closed measurable
 $\cald$-pullback absorbing set
 $K =\{ K(\tau, \omega): \tau \in \R,  \omega \in \Omega\}$
 $\in \cald$. 
\end{lem}

\begin{proof}
Given $\tau \in \R$    and $\omega \in \Omega$,   denote by 
  \be
 \label{plematt2_1}
 K(\tau, \omega) = \{ u\in H : \|u \|^2
 \le  M(\tau, \omega) \}, 
 \ee
   where
 $M(\tau, \omega)$ is  given by
  \be
 \label{plematt2_2}
 M(\tau, \omega) = z^{-2} (\tau, \theta_{2, -\tau} \omega )
 + {\frac 2\nu}z^{-2} (\tau, \theta_{2, -\tau} \omega )
 \int_{-\infty}^{\tau }
 e^{\nu\lambda (r-\tau)}
 z^2(r, \theta_{2,-\tau} \omega) \| f(\cdot, r)\|_{V^*}^2 dr.
\ee 
Since  for each $\tau \in \R$,  $M(\tau, \cdot): \Omega \to \R$
is   $(\calf, \calb(\R))$-measurable,   we know that
$K(\tau, \cdot): \Omega \to 2^H$ is a  measurable set-valued mapping.
It follows  from Lemma \ref{lematt1} that,  for each
$\tau \in \R $, $\omega \in \Omega$   and $D \in \cald$, 
  there  exists $T=T(\tau, \omega, D) >0$  such that   for   all
  $t \ge T$,
  $$
  \Phi (t, \tau-t, \theta_{2, -t} \omega, D(\tau -t, \theta_{2, -t} \omega ))
  =u (\tau,  \tau-t, \theta_{2, -\tau} \omega, D(\tau -t, \theta_{2, -t} \omega ))
  \subseteq K(\tau, \omega).
  $$
  Therefore,  $K =\{ K(\tau, \omega): \tau \in \R, \omega \in \Omega \}$
  will be a closed measurable  $\cald$-pullback absorbing set of $\Phi$ in $H$
  if one can show   that $K$ belongs to $\cald$. 
  For each $\tau \in \R$, $\omega \in \Omega$ and $r>0$,
  by \eqref{plematt2_1}   we have
   $$ \| K(\tau -  r, \theta_{2, -r} \omega) \|
  \le  {\frac 1{z(\tau -r, \theta_{2,-\tau} \omega)}}
  \sqrt{
   1+ {\frac 2\nu} \int_{-\infty} ^{\tau -r}
   e^{\nu\lambda (s -\tau + r)} z^2(s, \theta_{2, -\tau} \omega)
   \| f(\cdot, s ) \|^2_{V^*} ds
  }
  $$
$$
  \le e^{-\alpha \omega (-\tau)} 
  e^{\alpha \omega (-r)}
  \sqrt{
   1+ {\frac 2\nu} \int_{-\infty} ^{0}
   e^{\nu\lambda s } z^2(s+\tau -r, \theta_{2, -\tau} \omega)
   \| f(\cdot, s+\tau -r ) \|^2_{V^*} ds
  }
  $$
   \be\label{plematt2_3}
   \le e^{-\alpha \omega (-\tau)} 
  e^{\alpha \omega (-r)}
   \left ( 1+
  \sqrt{
   {\frac 2\nu} \int_{-\infty} ^{0}
   e^{(\nu\lambda-\delta) s }
    z^2(s+\tau -r , \theta_{2, -\tau} \omega)
   e^{\delta s} \| f(\cdot, s+\tau - r ) \|^2_{V^*} ds
  } \ 
  \right ).
  \ee
  Let $c$ be an arbitrary positive number
  and $\varepsilon 
  =\min \{ \nu\lambda -\delta,  \ {\frac 12} c\}$. 
  By \eqref{aspomega} we see   that  there
  exists $N_1>0$  such that  
  \be\label{plematt2_5}
  | -2 \alpha  \ \omega (p) |
  \le -\varepsilon p \quad
  \mbox{ for all }  \ p \le -N_1.
  \ee
  Let $s \le 0$   and $r \ge N_1$. Then
  $p = s-r \le  -N_1$ and hence 
  it follows  from \eqref{plematt2_5}
  that
   \be\label{plematt2_6}
   -2 \alpha  \ \omega (s-r) 
  \le -\varepsilon  (s-r) ,
  \quad  \mbox{ for  all } \ 
  s\le 0 \ \mbox{ and }  \ r \ge N_1 .
  \ee
  By \eqref{plematt2_6}   we have,
  for  all $s \le 0$   and $r \ge N_1$, 
  \be
  \label{plematt2_7}
  e^{(\nu\lambda -\delta) s} 
  z^2(s+ \tau -r, \theta_{2, -\tau} \omega)
  \le
  e^{(\nu\lambda -\delta) s} e^{2 \alpha  \omega (-\tau)}
  e^{ -2 \alpha \omega (s-r)}
  \le e^{ 2 \alpha \omega (-\tau)} e^{\varepsilon r}.
  \ee
  From \eqref{plematt2_3},
  \eqref{plematt2_5}  and \eqref{plematt2_7} we get
  that,  for all $r \ge   N_1$,
  $$ \| K(\tau -  r, \theta_{2, -r} \omega) \|
  \le
  e^{\varepsilon r -\alpha \omega (-\tau) }
  +\sqrt{{\frac 2\nu}}  e^{ {\frac 32} \varepsilon r}
  \left (
  \int_{-\infty}^0
  e^{\delta s} \| f(\cdot, s+\tau - r ) \|^2_{V^*} ds
  \right )^{\frac 12} 
  $$
 \be\label{plematt2_10}
   \le
  e^{{\frac 12 } c r -\alpha \omega (-\tau) }
  +\sqrt{{\frac 2\nu}}  e^{ {\frac 34} c r}
  \left (
  \int_{-\infty}^0
  e^{\delta s} \| f(\cdot, s+\tau - r ) \|^2_{V^*} ds
  \right )^{\frac 12},
 \ee
  where we have used the fact $\varepsilon \le {\frac 12} c$.
  It follows   from \eqref{plematt2_10}   that,   for  all 
  $r \ge N_1$,
 $$
  e^{-c  r} \| K(\tau -  r, \theta_{2, -r} \omega) \|
   \le
  e^{-{\frac 12 } c r -\alpha \omega (-\tau) }
  +\sqrt{{\frac 2\nu}}  e^{- {\frac 14} c r}
  \left (
  \int_{-\infty}^0
  e^{\delta s} \| f(\cdot, s+\tau - r ) \|^2_{V^*} ds
  \right )^{\frac 12}
  $$
 $$
 \le
  e^{-{\frac 12 } c r -\alpha \omega (-\tau) }
  +\sqrt{{\frac 2\nu}}  e^{- {\frac 14} c \tau}
  \left ( e^{{\frac 12} c (\tau -r)}
  \int_{-\infty}^0
  e^{\delta s} \| f(\cdot, s+\tau - r ) \|^2_{V^*} ds
  \right )^{\frac 12},
  $$
 which along with \eqref{fcond2} shows   that
 for every positive constant $c$,
 $$
 \lim_{r \to \infty}
 e^{-c  r} \| K(\tau -  r, \theta_{2, -r} \omega) \|
 =0,
 $$
 and hence 
 $K = \{ K(\tau, \omega): \tau \in \R, \omega \in \Omega \}$
 is tempered.  This completes   the proof.
  \end{proof}

  We now prove   the $\cald$-pullback   asymptotic 
  compactness of solutions of the stochastic   equations
  \eqref{nse1}-\eqref{nse2}.

 \begin{lem}
 \label{lematt3}
 Suppose  \eqref{poincare}  and   \eqref{fcond2}  hold.
 Then the continuous cocycle $\Phi$ associated with
 problem \eqref{nse1}-\eqref{nse4}  
  is $\cald$-pullback
 asymptotically compact  in $H$,
that is, for  every $\tau \in \R$, $\omega \in \Omega$, 
 $D=\{D(\tau, \omega): \tau \in \R, \omega \in \Omega \}$
 $  \in \cald$,
and $t_n \to \infty$,
 $u_{0,n}  \in D(\tau -t_n, \theta_{2, -t_n} \omega )$,  the sequence
 $\Phi(t_n, \tau -t_n,  \theta_{2, -t_n} \omega,   u_{0,n}  ) $   has a
   convergent
subsequence in $H $.
\end{lem}

\begin{proof}
Since $D \in \cald$  and
$u_{0,n}  \in D(\tau -t_n, \theta_{2, -t_n} \omega )$, by 
the proof of Lemma  \ref{lematt1} we find   that
for   each $n \in \N$, 
$v_{0,n} = z(\tau - t_n, \theta_{2, -\tau} \omega ) u_{0, n}$
$\in {\tilde{D}} (\tau -t_n, \theta_{2, -t_n} \omega)$, 
where 
$\tilde{D} \in \cald$ is  the family defined by
\eqref{plematt1_2}. 
Then it follows   from Lemma \ref{lem4}   that
the sequence
$v(\tau, \tau -t_n, \theta_{2, -\tau} \omega, v_{0,n} )$
of solutions   of problem \eqref{v1}-\eqref{v4} 
has a convergent subsequence in $H$.
By \eqref{vu}    we have
$$
u(\tau, \tau -t_n, \theta_{2, -\tau} \omega, u_{0,n} )
=
{\frac 1{z(\tau, \theta_{2, -\tau} \omega )}}
v(\tau, \tau -t_n, \theta_{2, -\tau} \omega, v_{0,n} ),$$
and hence  the sequence 
$u(\tau, \tau -t_n, \theta_{2, -\tau} \omega, u_{0,n} )$
has a convergent subsequence in $H$.   This implies    
$\Phi(t_n, \tau -t_n,  \theta_{2, -t_n} \omega,   u_{0,n}  ) $   has a
   convergent
subsequence in $H $.
\end{proof}

  We  are now   in a position to  present
  the main result of the paper, that is,   the existence of 
  tempered pullback attractors for
  the stochastic  Navier-Stokes equations.  
  
  \begin{thm}
\label{thmnse1}
 Suppose  \eqref{poincare}  and   \eqref{fcond2}  hold.
 Then the continuous cocycle $\Phi$ associated with
 problem \eqref{nse1}-\eqref{nse4}  
   has a unique $\cald$-pullback attractor $\cala
   =\{\cala(\tau, \omega):
      \tau \in \R, \ \omega \in \Omega \} \in \cald$
 in $H$.  Moreover,
 for each $\tau  \in \R$   and
$\omega \in \Omega$,
\be\label{thm1_1}
\cala (\tau, \omega)
=\Omega(K, \tau, \omega)
=\bigcup_{B \in \cald} \Omega(B, \tau, \omega)
\ee
\be\label{thm1_2}
 =\{\psi(0, \tau, \omega): \psi \mbox{ is any  }  \cald {\rm -}
 \mbox{complete orbit of } \Phi\} .
\ee
\end{thm}

\begin{proof}
By Lemma \ref{lematt2} we know   that
$\Phi$ has a closed measurable $\cald$-pullback absorbing set
in $H$.  On the other   hand,  by Lemma \ref{lematt3} we know that
  $\Phi$ is   $\cald$-pullback 
asymptotically  compact.
Then it follows    from Proposition \ref{att}
 that  $\Phi$   has a unique
 $\cald$-pullback attractor 
 $\cala$ in $H$
 and the structure of 
 $\cala$  is given by
 \eqref{thm1_1}-\eqref{thm1_2}.
\end{proof}

    We now  discuss    the existence of periodic 
    pullback   attractors   for problem \eqref{nse1}-\eqref{nse4}.
    Suppose   $f: \R \to V^*$
  is  a  periodic   function  
with period $T>0$.
If, in addition,  $f \in L^2_{loc} (\R, V^*)$,   then
  one   can verify    that $f$  satisfies  
  \eqref{fcond2}
  for  any $\delta>0$.
  In this case, 
for every ${\tilde{u}} \in H$,
 $t \ge 0$, $\tau \in \R$ and $\omega \in \Omega$,
   we  have    that 
$$
\Phi (t, \tau +T, \omega, {\tilde{u}} )
= u(t+ \tau +T, \tau +T,  \theta_{2, -\tau -T} \omega, {\tilde{u}})
=u(t +\tau, \tau, \theta_{2, -\tau} \omega, {\tilde{u}} ).
= \Phi (t, \tau,  \omega, {\tilde{u}} ). $$
By  Definition \ref{ds1}, we find   that
  $\Phi$ is  periodic with period  $T$.
Let $D \in \cald$ and $D_T$  be the $T$-translation
of $D$.  Then for every
$c>0$,  $s  \in \R$   and $\omega \in \Omega$,
\be\label{tranrde1}
\lim_{r \to \infty}
e^{-cr } \| D(s -r , \theta_{2, -r} \omega )\|^2 =0.
\ee
In particular, for $s  = \tau +T$   with $\tau \in \R$,  we get
from \eqref{tranrde1}  that 
\be\label{tranrde2}
\lim_{r \to  \infty}
e^{-cr } \| D_T(\tau -r, \theta_{2, -r} \omega )\|^2  
=
\lim_{r \to \infty}
e^{-cr } \| D(\tau +T-r, \theta_{2, -r} \omega )\|^2 =0.
\ee
From \eqref{tranrde2}  we see   that
$D_T \in \cald$,    and hence
$\cald$  is $T$-translation closed.
Similarly, one may check  that  
$\cald$  is also $-T$-translation closed.
Therefore 
   we find   
that  $\cald$  is $T$-translation invariant.
By   Proposition \ref{periodatt}, 
   the periodicity of the 
$\cald$-pullback attractor of 
problem \eqref{nse1}-\eqref{nse4}   follows.

  \begin{thm}
\label{thmnse2}
Let $f: \R \to V^*$  be    periodic  
with period $T>0$ and 
   $ f \in  L^2 ((0,T), V^*)$.
   If    \eqref{poincare} holds,  
 then the continuous cocycle $\Phi$ associated with
 problem \eqref{nse1}-\eqref{nse4}  
   has a unique   $\cald$-pullback attractor $\cala \in \cald$
 in $H$, which is periodic with period $T$.
\end{thm}

In the present
 paper,   we  have  discussed   the  pullback
attractors of the two-dimensional 
stochastic Navier-Stokes
equations with  non-autonomous deterministic 
force.  It  is
also interesting to 
consider the same problem for 
  the three-dimensional 
Navier-Stokes
equations,  where the uniqueness of solutions does not hold
anymore. 
In this case,   the author believes  that
the idea of multivalued  dynamical systems
developed in \cite{car6} can be 
extended to study the pullback attractors
of  the three-dimensional equations
with non-autonomous deterministic force.
The author will pursue this  line of   research
in the future.


\begin{thebibliography}{99}  

\bibitem{arn1}
L. Arnold, {\em Random Dynamical Systems}, Springer-Verlag, 1998.

 
  

\bibitem{bab1}
A.V. Babin and M.I. Vishik,
{\em Attractors of Evolution Equations}, North-Holland,
Amsterdam, 1992.

 
 


\bibitem{bal1}
J.M.   Ball,  Continuity properties and global
attractors of generalized semiflows and the Navier-Stokes
equations, {\em  J. Nonl. Sci.},  {\bf 7} (1997),
475-502.


\bibitem{bat1}
P.W. Bates,  H. Lisei and  K.  Lu,
 Attractors for stochastic lattice
 dynamical systems,
{\em Stoch. Dyn.},   {\bf 6}  (2006),      1-21.




\bibitem{bat2}
P.W. Bates,   K.  Lu   and B. Wang,
 Random attractors for  stochastic reaction-diffusion equations
on unbounded domains,  {\em J. Differential Equations},
 {\bf  246}   (2009),   845-869.

 
 \bibitem{brz1}
 Z. Brzezniak and Y. Li,
 Asymptotic compactness and
  absorbing sets for 2D
   stochastic Navier-Stokes equations on some unbounded domains,
   {\em    Transactions of  American Mathematical Society},
   {\bf 358} (2006),  5587-5629.




\bibitem{car2}
T. Caraballo, J. Real  and  I.D. Chueshov,
 Pullback attractors for stochastic heat
 equations in materials with memory,
  {\em Discrete Continuous Dynamical Systems B},
  {\bf 9 }  (2008),   525-539.


 
\bibitem{car3}
T. Caraballo, G. Lukaszewicz and J. Real,
Pullback attractors for asymptotically compact
non-autonomous dynamical systems,
{\em Nonlinear Analysis, TMA},  {\bf 64} (2006), 484-498.


\bibitem{car4}
T. Caraballo, G. Lukaszewicz and J. Real,
Pullback attractors for
non-autonomous 2D-Navier-Stokes equations
in some unbounded domains,
{\em  C. R. Acad. Sci. Paris I},  {\bf 342} (2006), 263-268.


 \bibitem{car6}
 T. Caraballo,  J. A.  Langa, V. S. Melnik and J. Valero,
 Pullback attractors of non-autonomous  and stochastic
 multivalued  dynamical
 systems,
 {\em Set-Valued Analysis}, {\bf   11}  (2003),
 153-201.



\bibitem{chu2}
I. Chueshov and M. Scheutzow,
On the structure of attractors  and invariant measures for a class of
monotone  random systems,
{\em Dynamical Systems}, {\bf 19} (2004), 127-144.





    \bibitem{cra1}
    H. Crauel,  A. Debussche and
    F. Flandoli, Random attractors,
    {\em J. Dyn. Diff. Eqns.}, {\bf 9} (1997), 307-341.





    \bibitem{cra2}
    H. Crauel  and
    F. Flandoli,  Attractors for random dynamical systems,
    {\em Probab. Th. Re. Fields}, {\bf 100} (1994), 365-393.

  
     \bibitem{fla1}
    F. Flandoli and B. Schmalfu$\beta$,
Random attractors for
the 3D stochastic Navier-Stokes equation with multiplicative
noise,
    {\em Stoch. Stoch. Rep.}, {\bf 59} (1996),  21-45.

 


 \bibitem{hal1}
  J.K.  Hale,   Asymptotic Behavior of
Dissipative Systems,
American Mathematical Society,
  Providence, RI, 1988.

     

\bibitem{huang1}
J. Huang and W. Shen,
Pullback attractors for nonautonomous and random parabolic equations on non-smooth domains,
{\em Discrete and Continuous Dynamical Systems},
{\bf 24} (2009),   855-882.





\bibitem{kloe1}
P.E. Kloeden and J.A. Langa,
 Flattening, squeezing and the existence of
random attractors,
{\em Proc. Royal Soc. London Serie A.}, {\bf 463}  (2007), 163-181.


 \bibitem{ros1}
 R. Rosa,
 The global attractor for the
 2D Navier-Stokes flow on some unbounded domains,
 {\em Nonlinear Analysis, TMA}, 
 {\bf 32}  (1998), 71-85.


\bibitem{schm1}
B.  Schmalfu$\beta$, Backward cocycles  
and attractors  of stochastic differential equations,
{\em International Seminar on Applied
 Mathematics-Nonlinear Dynamics: 
 Attractor Approximation and Global Behavior}, 
  1992,  185-192.


   
  \bibitem{sel1}
 R. Sell and  Y. You,
 Dynamics of Evolutionary Equations,
 Springer-Verlag,
New York, 2002.
  

\bibitem{tem1}
 R.  Temam,    Infinite-Dimensional Dynamical
Systems in Mechanics and Physics,
  Springer-Verlag,
  New York, 1997.


 
 

 \bibitem{wan2}
 B.  Wang,
Asymptotic behavior of stochastic wave equations with critical
exponents on $\R^3$,  
{\em  Transactions of  American Mathematical Society}, 
{\bf 363} (2011), 3639-3663.
 
 
 \bibitem{wan3}
 B. Wang,
 Random Attractors for the Stochastic 
 Benjamin-Bona-Mahony Equation on
  Unbounded Domains,
{ \em J. Differential Equations},  
{\bf 246} (2009), 2506-2537.


\bibitem{wan4}
       B. Wang,
      Sufficient and necessary criteria for
      existence of pullback attractors for
      non-compact random dynamical systems,
      arXiv:1202.2390v1 [math.AP], 2012.
     
 

 \end{thebibliography}
\end{document}